\documentclass[
a4paper,
oneside,
reqno,
]{amsart}

\usepackage{amsmath,amssymb,url}
\usepackage{enumitem} %%% Use the macros provided by this package when list labels stick out into the left margin. See the list in Section 7.

\usepackage{calc} % for \widthof{ <some expression> }
\usepackage{geometry}
\usepackage[T1]{fontenc}
\usepackage{lmodern} % for smooth fonts, compatible with T1
\usepackage{lipsum}

\usepackage{mathrsfs}
\usepackage[usenames,dvipsnames]{xcolor}
\usepackage[
colorlinks=true,
]{hyperref}
\usepackage{algorithm}
\usepackage{algpseudocode}
\usepackage{tikz}

\numberwithin{equation}{section}

\theoremstyle{plain}

\theoremstyle{definition}

\DeclareMathOperator{\Forall}{\forall}
\DeclareMathOperator{\Exists}{\exists}
\DeclareMathOperator{\vg}{Var}

\newcommand{\cupz}{\mathbin{\mathring\cup}}
\DeclareMathOperator{\clo}{Clo}
\newcommand{\pr}{e}
\newcommand{\imt}{IMT}

\newcommand{\algb}[1]{\langle #1 \rangle}

\newcommand{\N}{\mathbb{N}}

\newcommand{\alg}[1]{\mathbf{#1}}

\DeclareMathOperator{\algfn}{\alg{FN}}

\newcommand{\rg}[1]{[#1]}

\newtheorem{thm}{Theorem}[section]

\newtheorem{lma}[thm]{Lemma}

\newcommand{\enumref}[1]{(\ref{#1})}

\newcommand{\edgerel}{\mathrel{E}}

\theoremstyle{definition}

\theoremstyle{remark}

\begin{document}

%%%%%%%%%%%%%%%%%%%%%%%%%%%%%%%%%%%%%%%%%%%%%%%%%%%%%%%%%%%%%%%%%%%%%%
%% FRONT MATTER
%%%%%%%%%%%%%%%%%%%%%%%%%%%%%%%%%%%%%%%%%%%%%%%%%%%%%%%%%%%%%%%%%%%%%%

\title[Not all nilpotent monoids are finitely related]{Not all nilpotent monoids are finitely related}

%% For a single-authored paper, please use .
%% For multiple-authored papers, the line above MUST NOT be commented out
%% and $^*$ should be attached to the end of the name of the corresponding
%% author, as in the example below. 

%% The $^*$ on the name of the corresponding author on the submitted 
%% version of all multiple-authored papers is required by new EU data 
%% protection regulations. In line with the policy of Algebra Universalis 
%% since its inception, the $^*$ will not appear on the published version 
%% of the paper.

%% First author: in the order \author, \address, \urladdr, \email
%% For a sole-authored paper, use the \author[]{} command. 
%% For a multiple-authored paper, use the command \corrauthor[]{} for the corresponding author.
%% The corresponding author does not have to be the first-named author.
\author{Markus Steindl}
\address{RISC Software GmbH\\4232 Hagenberg\\Austria}  %
% \urladdr{https://www.risc-software.at/en/}
\email{ma.steindl@gmail.com}

% %% Second author: in the order \author, \address, \urladdr, \email
% \author[I. M. Coauthor]{Ivan M. Coauthor}
% \address{77 Sunset Strip\\Elk Haunches\\Alberta H0H 0H0\\Canada}
% \email{coauthor@eh.com.ca}

%% Thanks (Optional)
\thanks{Supported by the University of Colorado Boulder and the Austrian Science Fund (FWF) under grant no.~P27600}

%% AMS subject classification; see http://www.ams.org/msc
%% List classification codes in order of relevance

%% Key words and phrases

\begin{abstract}
A finite semigroup is finitely related (has finite degree) if its term functions are determined by a finite set of finitary relations.
For example, it is known that all nilpotent semigroups are finitely related.
A nilpotent monoid is a nilpotent semigroup with adjoined identity.
We show that every $4$-nilpotent monoid is finitely related.
We also give an example of a $5$-nilpotent monoid that is not finitely related.
To our knowledge, this is the first example of a finitely related semigroup where adjoining an identity yields a semigroup which is not finitely related.
We also provide examples of finitely related semigroups which have subsemigroups, homomorphic images, and in particular Rees quotients, that are not finitely related.
\end{abstract}

\maketitle

%%%%%%%%%%%%%%%%%%%%%%%%%%%%%%%%%%%%%%%%%%%%%%%%%%%%%%%%%%%%%%%%%%%%%%
%% MAIN MATTER
%%%%%%%%%%%%%%%%%%%%%%%%%%%%%%%%%%%%%%%%%%%%%%%%%%%%%%%%%%%%%%%%%%%%%%

\section{Introduction}
\label{sec:intro}

The study of finitely related algebraic structures is an area of crucial interest in the field of universal algebra.
Historically, finite relatedness has its roots in the general theory of clones \cite{szendrei1986clones}. 
However, its significance became more apparent in the study of the computational complexity of the constraint satisfaction problem (CSP).
In particular, the so-called algebraic approach~\cite{bulatov2008recent} associates with each finite relational constraint language a universal algebra, and (up to term equivalence), the finitely related algebras are precisely the class of algebras that arise in this way.
The classification of finitely related algebras generating congruence distributive varieties \cite{B13} and eventually, congruence modular varieties \cite{Bar18}, were  landmark contributions in this regard.
For semigroups, extensive characterizations of finite relatedness were presented in~\cite{DJPS11, May13, Dol18}.
We build upon these foundational contributions in the present paper.

A \emph{clone} $C$ on a set $A$ is a subset of $\{ f \colon A^n \to A \mid n \in \mathbb{N} \}$ which contains all projections
$\pr_i^n \colon (x_1,\dotsc,x_n) \mapsto x_i$ for $1 \le i \le n$, and which is closed under composition, i.e.,\ for $n$-ary $f \in C$ and $k$-ary $g_1,\dotsc,g_n \in C$, the composition
\[ 
f(g_1,\dotsc,g_n) \colon A^k \to A, \quad x \mapsto f(g_1(x),\dotsc,g_n(x)) 
\]
is also in $C$.
For a set $A$ and $k \in \mathbb{N}$, a subset $R \subseteq A^k$ is a $k$-ary \emph{relation} on $A$.
We say $f \colon A^n \to A$ \emph{preserves} $R$ if $f(r_1,\dotsc,r_n) \in R$ whenever $r_1,\dotsc,r_n \in R$, where $f$ is applied componentwise. 
A clone $C$ is \emph{finitely related} if there are finitely many relations $R_1,\dotsc,R_m$ on $A$ such that the elements of $C$ are precisely the functions that preserve $R_1,\dotsc,R_m$.
Note that every clone on a finite set $A$ can be described by a (possibly infinite) set of finitary relations by a standard result from clone theory.
An \emph{algebraic structure} (or just \emph{algebra}) $\alg{A} = \algb{A,F}$ is a nonempty set $A$ together with a set $F$ of finitary functions on $A$.
For algebras and terms in general see~\cite{BS81}.
In the case of semigroups, a \emph{term} $t(x_1,\dotsc,x_n)$ is a word in the alphabet $x_1,\dotsc,x_n$. For a fixed algebra $\alg{A}$, a term $t(x_1,\dotsc,x_n)$ induces an $n$-ary \emph{term function} $t^{\alg{A}} \colon A^n \to A$ by evaluation.
The set of all finitary term functions on $\alg{A}$ forms the \emph{clone of term functions} of $\alg{A}$ and is denoted by $\clo \alg{A}$. An algebra $\alg{A}$ is \emph{finitely related} if $\clo \alg{A}$ is finitely related.

In~\cite{AMM14} Aichinger, Mayr, and McKenzie showed that every finite algebra with few subpowers is finitely related. 
An algebra $\alg{A}$ is said to have few subpowers if there is a polynomial $p$ such that $\alg{A}^n$ has at most $2^{p(n)}$ subalgebras. 
This includes all finite Mal'cev algebras such as finite groups and rings.
Conversely, Barto proved Valeriote's conjecture~\cite{Bar18}: 
Every finite, finitely related algebra which generates a congruence modular variety has few subpowers.
In~\cite{Mor19} Moore showed that the question whether a given finite algebra is finitely related is undecidable. 
As a result, characterizations based on finite relatedness gain greater significance.
For further developments on finitely related algebras see also~\cite{Aic10, MMM12, B13, G17, BB17, S19}.

The characterization of finitely related semigroups began in~\cite{DJPS11}.
Davey, Jackson, Pitkethly, and Szab\'o showed that nilpotent semigroups and commutative semigroups are finitely related. 
They also provided examples of finitely related algebras which have homomorphic images, subalgebras, and direct products, respectively, that are not finitely related.
Mayr~\cite{May13} proved that all Clifford semigroups, regular bands (a variety of idempotent semigroups), $3$-nilpotent monoids, and semigroups with a single idempotent element are finitely related.
He also gave the first known example of a semigroup that is not finitely related, the $6$-element Brandt monoid.
In~\cite{Dol18} Dolinka showed that all finite bands are finitely related.

In this paper, we address the following questions.
\begin{itemize}
\item Up to what $n \in \N$ are all $n$-nilpotent monoids finitely related?
\item Are all subsemigroups, homomorphic images, and direct products of a finitely related semigroup also finitely related?
\item If we adjoin an identity to a finitely related semigroup, is the resulting monoid finitely related?
\end{itemize}

\section{Results}

For a set $A$, $n \in \mathbb{N}$, and $f \colon A^n \to A$, we define new functions by identifying arguments. 
For $i,j \in \{1,\dots,n\}$ with $i < j$ let
\[
f_{ij} \colon A^n \to A, \quad 
(x_1,\dotsc,x_n) \mapsto f(x_1,\dots,x_{i-1},x_j,x_{i+1},\dots,x_n). 
\]
The function $f_{ij}$ is called an \emph{identification minor} of $f$.
For an algebraic structure $\alg{A}$ we say $f \colon A^n \to A$ has \emph{\imt} if every identification minor $f_{ij}$ is a term function of $\alg{A}$.

The following lemma reduces the question whether a finite semigroup is finitely related to a syntactic condition. 
Instead of explicitly listing the relations describing the clone of term functions, it suffices to show that all functions with \imt{} above a certain arity are term functions.

\begin{lma}[{\cite{J74,RS83}}]\label{lma:main_cond_fin_rel}
For a finite algebra $\alg{A}$ the following are equivalent:
\begin{enumerate}
\item\label{it1_lma:main_cond_fin_rel}
$\alg{A}$ is finitely related.
\item\label{it2_lma:main_cond_fin_rel}
$\Exists k \in \mathbb{N} \Forall n > k:$ every $f \colon A^n \to A$ with \imt{} is a term function of $\alg{A}$.
\end{enumerate}
\end{lma}

\noindent
In~\cite{RS83} Rosenberg and Szendrei define the \emph{degree} of $\alg{A}$ to be the least $k$ that satisfies condition~
\enumref{it2_lma:main_cond_fin_rel} of Lemma~\ref{lma:main_cond_fin_rel}. The degree is finite if and only if $\alg{A}$ is finitely related.

For $d \in \mathbb{N}$ a semigroup $\alg{S}$ is \emph{$d$-nilpotent} if 
\[ \text{$x_1 \dotsm x_d = y_1 \dotsm y_d$\quad{}for all $x_1,\dotsc,x_d,y_1,\dotsc,y_d \in S$.}\]
$\alg{S}$ is \emph{nilpotent} if it is $d$-nilpotent for some $d \in \mathbb{N}$.
Every nilpotent semigroup has a unique zero element, denoted by $0$.
Every finite nilpotent semigroup $\alg{S}$ is $|S|$-nilpotent by~\cite[Theorem~3.3]{DJPS11}.
The following is a consequence of a result by Willard~\cite[Lemma 1.2]{W96}. 
%
% \begin{lma}[\cite{}]
% Let $\alg{A}$ be a finite algebra. 
% Assume that there exists $k \in \mathbb{N}$ such that every term function on $\alg{A}$ depends on at most $k$ arguments. 
% Then $\alg{A}$ is finitely related with degree at most $\max(|A|, k+2)$.
% \end{lma}
%
\begin{thm}[{\cite[Theorem 3.4]{DJPS11}}]
Every finite nilpotent semigroup $\alg{S}$ is finitely related with degree at most $|S|+1$.
\end{thm}

A \emph{$d$-nilpotent monoid} is a $d$-nilpotent semigroup with adjoined identity.
Mayr showed in~\cite{May13} that every finite $3$-nilpotent monoid $\alg{S}$ is finitely related with degree at most $\max(|S|+1,4)$. We strengthen his result as follows.
\begin{thm}\label{thm:4nilp_fin_rel}
Every finite $4$-nilpotent monoid $\alg{S}$ is finitely related with degree at most $|S| + 1$.
\end{thm}

A subsemigroup $\alg{I}$ is an \emph{ideal} of $\alg{S}$ if for all $i \in I$ and $s \in S$, both $s i \in I$ and $i s \in I$.
Every ideal $\alg{I}$ of $\alg{S}$ induces a congruence $\theta = (I \times I) \cup \{(s,s) \mid s \in S\}$ which identifies the elements in $I$. We refer to the quotient $\alg{S} \mathbin{/} \theta$ as the \emph{Rees quotient} of $\alg{S}$ by $\alg{I}$.

For a nonempty set $A$, the notation $\alg{A}^+$ denotes the \emph{free semigroup} on $A$.
% For $d \in \mathbb{N}$, we define $\algfn_d A$ to be the Rees quotient of $\alg{A}^+$ by the ideal
% \[ I = \{ u \in A^+ \mid \text{$u$ has length at least d} \} \]
% and call it the \emph{free $d$-nilpotent semigroup} on $A$.
For $d \in \mathbb{N}$, we define the \emph{free $d$-nilpotent semigroup} on $A$, denoted by $\algfn_d A$, as the Rees quotient of $\alg{A}^+$ with respect to the ideal
\[ I = \{ u \in A^+ \mid \text{the length of \( u \) is at least \( d \)} \}. \]

If we adjoin an identity to a free nilpotent semigroup, then the resulting monoid is also finitely related:

\begin{thm}\label{thm:FN1_finitely_related}
For a nonempty finite set $A$ and $d \in \mathbb{N}$, $\alg{S} := (\algfn_d A)^1$ is finitely related with degree at most $|S| + 1$.
\end{thm}

The following nilpotent monoid is not finitely related.

\begin{thm}\label{thm:FN5_factor_not_fr}
Let $\alg{S} := (\algfn_5 \{a,b\} \mathbin{/}\theta)^1$, where $\theta$ is the congruence whose equivalence classes are $\{ (ab)^2, (ba)^2 \}$, $\{ a^2b^2, b^2a^2 \}$, and the rest singletons.
Then $\alg{S}$ is not finitely related.
\end{thm}

We can now answer the questions from the introduction.
The $5$-nilpotent semigroup $\algfn_5 \{a,b\} \mathbin{/}\theta$ is finitely related. 
Adjoining an identity yields a $5$-nilpotent monoid that is not finitely related.
In addition, $(\algfn_5 \{a,b\})^1$ is finitely related by Theorem~\ref{thm:FN1_finitely_related}. 
Its homomorphic image $(\algfn_5 \{a,b\} \mathbin{/}\theta)^1$ is not finitely related. 
To the best of our knowledge, these are the first examples of such semigroups.

For a signature $\mathcal{F}$ and terms $s(x_1,\dotsc,x_n)$ and $t(x_1,\dotsc,x_n)$ over $\mathcal{F}$ an \emph{identity} is an expression of the form $s\approx t$.
An algebra $\alg{A}$ over $\mathcal{F}$ \emph{satisfies} the identity $s \approx t$ if $s^{\alg{A}}(a_1,\dotsc,a_n) = t^{\alg{A}}(a_1,\dotsc,a_n)$ for all $a_1,\dotsc,a_n \in A$.
A class $V$ of algebras over a signature $\mathcal{F}$ is called a \emph{variety} if
there is a set $\Sigma$ of identities over $\mathcal{F}$ such that $V$ contains precisely the algebras that satisfy all identities in $\Sigma$.
By Birkhoff's HSP theorem a class of algebras is a variety if and only 
if it is closed under homomorphic images, subalgebras, and direct products.
For a variety $V$, $V_\textnormal{fin}$ denotes the class that contains all finite members of $V$.

\begin{thm}\label{thm:fin_rel_propagation_sg}
For a variety $V$ of semigroups, the following are equivalent:
\begin{enumerate}
\item\label{it1_thm:fin_rel_propagation_sg}
$\Forall \alg{S}, \alg{T} \in V_\textnormal{fin}$: if $\alg{S} \times \alg{T}$ is finitely related (f.r.), then $\alg{S}$ is f.r.
\item\label{it2_thm:fin_rel_propagation_sg}
$\Forall \alg{S} \in V_\textnormal{fin}$: if $\alg{S}$ is f.r., then every homomorphic image of $\alg{S}$ is f.r.
\setcounter{enumi}{2}
\item\label{it3_thm:fin_rel_propagation_sg}
$\Forall \alg{S} \in V_\textnormal{fin}$: if $\alg{S}$ is f.r., then every subsemigroup of $\alg{S}$ is f.r.
\item\label{it4_thm:fin_rel_propagation_sg}
$\Forall \alg{S} \in V_\textnormal{fin}$: if $\alg{S}$ is f.r., then every Rees quotient of $\alg{S}$ is f.r.
\end{enumerate}
\end{thm}

\noindent
The varieties $V$ that satisfy the conditions in Theorem~\ref{thm:fin_rel_propagation_sg} include those of commutative semigroups, bands, and $d$-nilpotent semigroups for a given $d \in \N$.
Conversely, the variety generated by $(\algfn_5 \{a,b\})^1$ does not meet these conditions by Theorem~\ref{thm:FN5_factor_not_fr}.
As a consequence of Theorem~\ref{thm:fin_rel_propagation_sg}, there exist finitely related semigroups with subsemigroups and Rees quotients that are not finitely related.
See Lemma~\ref{lma:examples} for an example.

% By Theorems~\ref{thm:FN5_factor_not_fr} and~\ref{thm:fin_rel_propagation_sg} there are finitely related semigroups which have subsemigroups and Rees quotients, respectively, that are not finitely related.
% See Lemma~\ref{lma:examples} for an example.

\section{Proofs}

\begin{thm}[{\cite[Theorem~3.6]{DJPS11}}]\label{thm:comm_monoid_fin_rel}
Every finite commutative monoid $\alg{S}$ is finitely related with degree at most $\max(|S|,3)$.
\end{thm}

\noindent
Note that the original result in $[8]$ only claims that the degree is at most $\max (|S|, 4)$, but the proof actually yields the slightly stronger version given above.

A function $f \colon A^n \to A$ \emph{depends on its $i$-th argument} if there are $a, b \in A^n$ such that $a_j = b_j$ for all $j \ne i$ and $f(a) \ne f(b)$. 
For $n \in \mathbb{N}$ we write $\rg{n}$ instead of $\{1,\dotsc,n\}$.
Now let $m \le n$ in $\mathbb{N}$ and $i_1 < \dotsb < i_m$ in $\rg{n}$. 
Let $y_1,\dotsc,y_m$ be elements of a monoid $\alg{S}$. 
For a tuple $x \in S^n$ with
\[
x_j = \begin{cases}
y_k & \text{if } j = i_k \text{ for some } k \in \rg{m}, \\
1 & \text{otherwise,}
\end{cases}
\]
we also write
\[
x = (\bar 1,\underset{i_1}{y_1},\bar 1,\underset{i_2}{y_2},\dotsc,\bar 1,\underset{i_m}{y_m},\bar 1).
\]
The arity of $x$ will be clear from the context.
For a term $t(x_1,\dotsc,x_n)$,
\[
t(\bar 1,\underset{i_1}{y_1},\bar 1,\underset{i_2}{y_2},\dotsc,\bar 1,\underset{i_m}{y_m},\bar 1)
\]
denotes the term obtained when $x_{i_k}$ is replaced by $y_k$ for $k \in \rg{m}$ and the remaining $x_j$ are removed. %set to $1$.

\begin{lma}[{\cite[Lemma 2.6]{May13}}]
\label{lma:mayr_exist_depending_minor}
Let $A$ be a finite set, $n>|A|+1$, and $f: A^n \rightarrow$ A. If $f$ depends on its $i$-th argument for some $i \in \rg{n}$, then there exist $\ell < m$ in $\rg{n} \backslash\{i\}$ such that $f_{\ell m}$ depends on its i-th argument.
\end{lma}

Note that the bound $|A| + 1$ cannot be reduced to $|A|$ in Lemma~\ref{lma:mayr_exist_depending_minor}.
Here is a simple counter example.
Let $A = \{0, 1\}$, $n = 3$, and $f \colon A^n \to A$ such that $f_{12} = f_{13} = x_3$ and $f_{23} = x_1$.
Verify that $f$ exists by creating the value table.
% f(0, 0, 0) = 0
% f(0, 0, 1) = 1
% f(0, 1, 0) = 0
% f(0, 1, 1) = 0
% f(1, 0, 0) = 1
% f(1, 0, 1) = 1
% f(1, 1, 0) = 0
% f(1, 1, 1) = 1
Then $f$ depends on $x_2$ since $f(0, 0, 1) = 1 \ne 0 = f(0, 1, 1)$.
However, none of the identification minors depends on $x_2$.

\begin{lma}\label{lma:ei_exists}
Let $\alg{S}$ be a finite monoid. 
Let $n > |S| + 1$ and let $f \colon S^n \to S$ have \imt{} and depend on its $i$-th argument. 
Then there exists $e_i \in \N$ such that
\[ 
\text{$x_i^{e_i}$ induces $f(\bar{1},\underset{i}{x_i},\bar{1})$.}
\]
We refer to the minimal such $e_i$ as the \emph{variable exponent} of $x_i$ in $f$.
\end{lma}

\begin{proof}
For $\mathbf{S}$ of size 1, the lemma holds trivially.
Assume $|S| \ge 2$.
By Lemma~\ref{lma:mayr_exist_depending_minor}, there exist $k < \ell$ in $\rg{n}\setminus\{i\}$ such that $f_{k \ell}$ depends on $x_i$.
Let $t$ be a term that induces $f_{k \ell}$.
Since $t^\alg{S}$ depends on its $i$-th argument, $x_i$ occurs $e_i \in \N$ times in $t(x_1, \dots, x_n)$.
We obtain
\[
  f(\bar{1},\underset{i}{x_i},\bar{1}) = f_{k \ell}(\bar{1},\underset{i}{x_i},\bar{1}) = t^\alg{S}(\bar{1},\underset{i}{x_i},\bar{1}) = x_i^{e_i},
\]
which proves the Lemma.
\end{proof}

\begin{lma}\label{lma:FN_occurrences}
Let $\alg{S}$ be a finite nilpotent monoid and $c \in \mathbb{N}$ minimal such that $S \setminus \{1\}$ satisfies $x^c \approx 0$.
Let $n > |S| + 1$ and let $f \colon S^n \to S$ have \imt{} and depend on all variables.
Let $e_1,\dotsc,e_n$ be the variable exponents of $f$.
Then each term $t$ inducing an identification minor $f_{ij}$ has the following properties:
\begin{enumerate}
  \item\label{it1_lma:FN_occurrences} 
  $x_i$ does not occur in $t$,
  \item\label{it2_lma:FN_occurrences}
  for $k\in \rg{n}\setminus \{i,j\}$, $x_k$ occurs $e_k$ times if $e_k < c$, and at least $c$ times otherwise,
  \item\label{it3_lma:FN_occurrences}
  $x_j$ occurs $e_i + e_j$ times if $e_i + e_j < c$, and at least $c$ times otherwise.
\end{enumerate}
\end{lma}

\begin{proof}
Let $f_{ij}$ be an identification minor with $i < j$ in $\rg{n}$ induced by some term $t$.
Since every nilpotent monoid has at least two elements by definition, we have $n \ge |S| + 2 \ge 4$.

\begin{enumerate}
\item[\enumref{it1_lma:FN_occurrences}]
Suppose $x_i$ occurs in $t$.
Then we obtain the contradiction
\[ 1 = f(1,\dotsc,1) = f_{ij}(\bar 1, \underset{i}{0}, \bar 1) = t^\alg{S}(\bar 1, \underset{i}{0}, \bar 1) = 0. \]

\item[\enumref{it2_lma:FN_occurrences}]
Let $k\in \rg{n}\setminus \{i,j\}$. 
Then
\[ \text{$t(\bar 1, \underset{k}{x_k}, \bar 1)$ induces $f(\bar 1, \underset{k}{x_k}, \bar 1)$.} \]
Note that $x,x^2,\dotsc,x^{c-1}$ are pairwise inequivalent terms in $\alg{S}$,
while $x^c,x^{c+1},\dotsc$ are all equivalent.
Thus, if $e_k < c$, then $x_k$ occurs precisely $e_k$ times in $t$.
If $e_k \ge c$, then $x_k$ occurs at least $c$ times in $t$.

\item[\enumref{it3_lma:FN_occurrences}]
Let $k < \ell$ in $\rg{n} \setminus \{i,j\}$ and $s$ be a term that induces $f_{k\ell}$. 
Then
\[ \text{$t(\bar 1, \underset{j}{x_j}, \bar 1)$ and $s(\bar 1, \underset{i}{x_j}, \bar 1, \underset{j}{x_j}, \bar 1)$ both induce $f(\bar 1, \underset{i}{x_j}, \bar 1, \underset{j}{x_j}, \bar 1)$.} \]
If $e_i + e_j < c$, then by~\enumref{it2_lma:FN_occurrences} $x_i$ and $x_j$ occur $e_i$ and $e_j$ times in $s$, respectively.
Thus $x_j$ occurs $e_i + e_j$ times in $t$.
If $e_i + e_j \ge c$, then $x_i$ and $x_j$ occur at least $\min(e_i, c)$ and $\min(e_j, c)$ times in $s$, respectively, again by~\enumref{it2_lma:FN_occurrences}.
Thus $x_j$ occurs at least $c$ times in $t$.

\end{enumerate}
\end{proof}

% % We did not need this lemma:
% \begin{lma}\label{lma:term_var_occs}
% Let $\alg{S}$ be a $d$-nilpotent monoid and let $t(x_1,\dotsc,x_n)$ be a term with $e_i$ occurrences of $x_i$ for $i \in \rg{n}$.
% Then there is a term $r \approx t$ such that $x_i$ occurs $min(e_i, d)$ times in $r$.
% \end{lma}
% \begin{proof}
% W.l.o.g.\ assume
% \begin{equation*}%\label{eq0_lma:nilp_mon_IMT-fct_form}
% e_1,\dotsc,e_r \le d-1,\quad e_{r+1},\dotsc,e_n \ge d
% \end{equation*}
% for some $r \in \{0, \dots, n\}$.
% Let $r := t(x_1,\dotsc,x_r,\bar 1) \,x_{r+1}^d \dotsm x_n^d$.
% Let $a \in S^n$. 
% If $r = n$ or $a_{r+1} = \dots = a_n = 1$, then $r^\alg{S}(a) = t^\alg{S}(a)$.
% Otherwise, $r^\alg{S}(a) = 0 = t^\alg{S}(a)$ by $d$-nilpotence.
% We showed $r \approx t$.
% \end{proof}

% % We did not need this lemma:
% \begin{lma}\label{lma:cardinality_of_noncommutative_and_nilpotent}
% Let $\alg{S}$ be a noncommutative semigroup and $d \in \N$ minimal such that $\alg{S}$ is $d$-nilpotent.
% Then $|S| \ge d + 1$.
% \end{lma}
% \begin{proof}
% By minimality of $d$ there are $s_1, \dots, s_{d-1} \in S$ such that $p := s_1 \cdots s_{d-1}$ is nonzero.
% Assume $\{s_1 \cdots s_i \mid i \in \rg{d}\}$ are the only elements of $S$.
% Then for $j \in \{2, \dots, d-1\}$ we have $s_j = s_1$, otherwise $p$ would be zero.
% However, this contradicts the commutativity of $\alg{S}$.
% \end{proof}

\begin{lma}\label{lma:depend_on_each_var}
Let $\alg{A}$ be a finite algebra, $n \ge 2$, and let $f \colon A^n \to A$ have \imt{}. 
If $f$ does not depend on one of its variables, then $f$ is a term function.
\end{lma}

\begin{proof}
W.l.o.g.\ we can assume that $f$ does not depend on $x_1$.
Then $f = f_{12}$, which is a term function.
\end{proof}

\begin{lma}\label{lma:FN_unique_term}
Let $d \ge 2$, let $A$ be a finite set with $|A| \ge 2$, and let $\alg{S} := (\algfn_d A)^1$.
Let $t(x_1,\dotsc,x_n)$ be a term with $e_i$ occurrences of $x_i$ for $i \in \rg{n}$.
\begin{enumerate}
\item\label{it1_lma:FN_unique_term}
If $e_1 + \dotsb + e_n < d$, then $t$ is the unique term inducing $t^\alg{S}$.
\item\label{it2_lma:FN_unique_term}
If $n = 2$ and $e_1 + e_2 \ge d$, then each permutation of $t$ induces $t^\alg{S}$.
\end{enumerate}
\end{lma}

\begin{proof}
\begin{enumerate}
\item[\enumref{it1_lma:FN_unique_term}]
Assume $e_1 + \dotsb + e_n < d$ and let $s$ be another term that induces $t^\alg{S}$.
Our aim is to show that $t = s$.
If $t \neq s$, then there exist terms $u, v, w$ and distinct indices $i,j \in \rg{n}$ such that
\[\begin{aligned}
  t &= u(x_1, \dots, x_n) \, x_i \, v(x_1, \dots, x_n) \\
  s &= u(x_1, \dots, x_n) \, x_j \, w(x_1, \dots, x_n) \text{.}
\end{aligned}\]
W.l.o.g. assume $i < j$.
Pick $a \neq b$ from $A$.
Then
\[\begin{aligned}
  t^{\alg{S}}(\bar{1},\underset{i}{a},\bar{1},\underset{j}{b},\bar{1})
  &= u^{\alg{S}}(\bar{1},\underset{i}{a},\bar{1},\underset{j}{b},\bar{1}) \, a \, v^{\alg{S}}(\bar{1},\underset{i}{a},\bar{1},\underset{j}{b},\bar{1}) \\
  &\neq u^{\alg{S}}(\bar{1},\underset{i}{a},\bar{1},\underset{j}{b},\bar{1}) \, b \, w^{\alg{S}}(\bar{1},\underset{i}{a},\bar{1},\underset{j}{b},\bar{1}) \\
  &= s^{\alg{S}}(\bar{1},\underset{i}{a},\bar{1},\underset{j}{b},\bar{1}) \\
\end{aligned}\]
by the definition of $(\algfn_d A)^1$ and since the left hand side of the inequality is ${} \neq 0$.
This contradicts the fact that $s$ induces $t^{\alg{S}}$.
Thus $t = s$.
\item[\enumref{it2_lma:FN_unique_term}]
Assume $n = 2$ and $e_1 + e_2 \ge d$. 
Let $s$ be a permutation of $t$ and $u, v \in S$.
If both $u$ and $v$ are distinct from $1$, then $t^\alg{S}(u, v) = 0 = s^\alg{S}(u, v)$ by $d$-nilpotence.
If $u = 1$, then $t^\alg{S}(u, v) = v^{e_2} = s^\alg{S}(u, v)$.
If $v = 1$, the equation follows similarly.
\end{enumerate}
\end{proof}

\begin{lma}\label{lma:cycle-free_digraph}
Let $\algb{V, \edgerel}$ be a finite, cycle-free digraph.
Then there exists a linear order on~$V$ that includes~$\edgerel$.
\end{lma}

% \begin{proof}
% Induction on $|V|$.
% The lemma is clear for $|V|=1$.
% Let $|V| \ge 2$ and assume that the lemma holds for every $k < |V|$.
% Pick $x \in V$ without incoming edge.
% By the induction hypothesis there is a linear order $\le$ on $V \setminus \{x\}$ that includes the restriction of $\edgerel$ to $V \setminus \{x\}$.
% Now let $x \le v$ for all $v \in V \setminus\{x\}$.
% \end{proof}

\begin{proof}
First observe that the reflexive transitive closure $S$ of $E$ is a partial order.
By the Szpilrajn extension theorem~\cite{szpilrajn1930extension}, there exists a linear order that includes $S$.
\end{proof}

% \begin{customthm}{\ref*{thm:FN1_finitely_related}}
% \textcolor{Gray}{
% For a nonempty finite set $A$ and $d \in \mathbb{N}$, $\alg{S} := (\algfn_d A)^1$ is finitely related with degree at most $\max(|S|, 3)$.}
% \end{customthm}

\newcommand{\vord}{\prec}

\begin{proof}[Proof of Theorem~\ref{thm:FN1_finitely_related}]
First assume $d \le 2$ or $|A| = 1$. 
Then $\alg{S}$ is commutative.
By Theorem~\ref{thm:comm_monoid_fin_rel} $\alg{S}$ is finitely related with degree at most $\max(|S|,3)$.
If $|S|=1$, then $\alg{S}$ is finitely related with degree $1$.
Thus the degree is at most $|S|+1$.

Now assume $d \ge 3$ and $|A| \ge 2$.
Let $n > |S| + 1$ and let $f \colon S^n \to S$ have \imt.
Note that $|S| \ge d + 1$ and thus $n \ge 6$.
By Lemma~\ref{lma:depend_on_each_var} we can assume that $f$ depends on all variables.
Let $e_1,\dotsc,e_n$ be the variable exponents of $f$ defined in Lemma~\ref{lma:ei_exists}.
We will construct a term $s(x_1,\dotsc,x_n)$ that induces $f$ with $e_i$ occurrences of $x_i$ for $i \in \rg{n}$.
We enumerate the occurrences of $x_i$ by a second index $\alpha \in \rg{e_i}$. 
To this end, let
\[ X := \{ x_{i\alpha} \mid i \in \rg{n},\, \alpha \in \rg{e_i} \}. \]
Let $i < j$ in $\rg{n}$ with $e_i + e_j < d$.
By \imt{} and Lemma~\ref{lma:FN_unique_term}
\begin{equation}\label{eq0_thm:FN1_finitely_related}
\text{there is a unique term $h_{ij}$ that induces $f(\bar 1,\underset{i}{x_i},\bar 1,\underset{j}{x_j},\bar 1)$}.
\end{equation}
By Lemma~\ref{lma:FN_occurrences} $h_{ij}$ has $e_i$ occurrences of $x_i$ and $e_j$ occurrences of $x_j$.
We define a strict linear order $<_{ij}$ on the subset $\{ x_{i \alpha} \mid \alpha \in \rg{e_i} \} \cup \{ x_{j \beta} \mid \beta \in \rg{e_j} \}$ of $X$.
For $k, \ell \in \{i, j\}$ let
\[
  x_{k \alpha} <_{ij} x_{\ell \beta}
\]
if the $\alpha$-th occurrence of $x_k$ precedes the $\beta$-th occurrence of $x_\ell$ in $h_{ij}$.
Let $\vord$ be the relation on $X$ such that
\[
{\vord} := \bigcup \{ {<_{ij}} \mid i < j \text{ in }\rg{n} \text{ and } e_i + e_j < d \}.
\]
We claim that
\begin{equation}\label{eq1_thm:FN1_finitely_related}
\algb{X,\vord} \text{ is cycle-free}.
\end{equation}
Suppose there is a cycle $c = (x_{i_1\alpha_1}, \dotsc, x_{i_m\alpha_m}, x_{i_1\alpha_1})$ of minimal length $m$. We show that 
\begin{equation}\label{eq2_thm:FN1_finitely_related}
i_1,\dotsc,i_m \text{ are distinct.}
\end{equation}
Suppose $i_u = i_v$ for $u < v$ in $\rg{m}$.
If $\alpha_u \le \alpha_v$, then $x_{i_{u-1}\alpha_{u-1}} \vord x_{i_v\alpha_v}$ and 
\[ (x_{i_1\alpha_1}, \dotsc, x_{i_{u-1}\alpha_{u-1}}, x_{i_v\alpha_v}, \dotsc, x_{i_m\alpha_m}, x_{i_1\alpha_1}) \] is a cycle shorter than $c$.
If $\alpha_v < \alpha_u$, then $x_{i_{v-1}\alpha_{v-1}} \vord x_{i_u\alpha_u}$ and 
\[ (x_{i_u\alpha_u}, \dotsc, x_{i_{v-1}\alpha_{v-1}}, x_{i_u\alpha_u}) \] is a cycle shorter than $c$. 
This proves~\eqref{eq2_thm:FN1_finitely_related}.
W.l.o.g.\ we can assume that $i_1=1,\dotsc,i_m=m$.

Let~$\oplus$ and~$\ominus$ be addition and subtraction modulo $m$ on the set $\rg{m}$, respectively.
By the definition of $\vord$ we have 
\begin{equation}\label{eq2.5_thm:FN1_finitely_related}
\text{$e_i + e_{i \oplus 1} < d$ for $i \in \rg{m}$.}
\end{equation}

Assume $m \le n-2$. 
Then $f(x_1,\dotsc,x_m, \bar 1) = f_{n-1, n}(x_1, \dotsc, x_m, \bar 1)$ is induced by some term $t$.
For all $i < j$ in $\rg{m}$ we have
\[
t(\bar 1,\underset{i}{x_i},\bar 1,\underset{j}{x_j},\bar 1) = h_{ij}
\]
by~\eqref{eq0_thm:FN1_finitely_related}.
But then the order of the variables in $t(x_1, \dots, x_m)$ prevents the cycle~$c$.

Now assume $m \ge n-1$.
Pick $i$ such that $e_i = \max(e_1, \dotsc, e_m)$.
By~\eqref{eq2.5_thm:FN1_finitely_related} we have $e_{i \ominus 1} + e_i < d$ and $e_i + e_{i\oplus 1} < d$. 
Thus $e_{i \ominus 1} + e_{i \oplus 1} < 2d - 2e_i$.
From the maximality of $e_i$ we obtain $e_{i \ominus 1} + e_{i \oplus 1} \le 2 e_i$.
Adding these inequalities yields $e_{i \ominus 1} + e_{i \oplus 1} < d$.
Thus 
$x_{i \ominus 1,\alpha_{i \ominus 1}}$ and $x_{i \oplus 1,\alpha_{i \oplus 1}}$ are related by $<_{i \ominus 1, i \oplus 1}$.
From that and $m \ge n-1 \ge 5$ it follows that we can reduce $c$ to a shorter cycle. 
This contradicts the minimality of $m$. 
We proved~\eqref{eq1_thm:FN1_finitely_related}.

By Lemma~\ref{lma:cycle-free_digraph} there is a linear order $<$ on $X$ that includes $\vord$.
We order the elements of $X$ by $<$ and drop the second index.
Let $t(x_1,\dotsc,x_n)$ be the resulting term.
We claim that
\begin{equation}\label{eq3_thm:FN1_finitely_related}
t \text{ induces } f.
\end{equation}
Fix $a \in S^n$.
Let $i_1 < \dotsb < i_k$ in $\rg{n}$ be the positions where $a$ is ${} \ne 1$.
First assume $e_{i_1} + \dotsb + e_{i_k} \ge d$.
Then $t^\alg{S}(a) = 0$ by $d$-nilpotence.
Since $n > |S|$, there are $i < j$ in $\rg{n}$ such that $a_i = a_j$.
Thus $f(a) = f_{ij}(a)$.
Note that $d \in \N$ is minimal such that $S \setminus \{1\}$ satisfies $x^d \approx 0$.
By Lemma~\ref{lma:FN_occurrences} the product $f_{ij}(a)$ has at least $d$ factors $\ne 1$. 
Hence $f(a) = f_{ij}(a) = 0 = t^\alg{S}(a)$.

Now assume $e_{i_1} + \dotsb + e_{i_k} < d$.
Note that $k < d < |S| < n$. 
Thus $n-k \ge 2$ arguments are $1$. 
By \imt
\[ f(\bar 1,\underset{i_1}{x_{i_1}},\bar 1,\underset{i_2}{x_{i_2}},\dotsc,\bar 1,\underset{i_k}{x_{i_k}},\bar 1) \text{ is induced by some term } s(x_{i_1},\dotsc,x_{i_k}). \]
We have $s(\bar 1,\underset{i}{x_{i}},\bar 1,\underset{j}{x_{j}},\bar 1) = h_{ij}$ for all $i < j$ in $\{i_1, \dots, i_k\}$.
This means the variables of $s$ are ordered by ${<_s} := \bigcup \{ {<_{ij}} \mid i < j \text{ in }\{i_1, \dots, i_k\} \}$.
Thus $<$ includes $<_s$, and the latter is also a linear order.
Therefore
\[
t(\bar 1,\underset{i_1}{x_{i_1}},\bar 1,\underset{i_2}{x_{i_2}},\dotsc,\bar 1,\underset{i_k}{x_{i_k}},\bar 1) = s.
\]
Hence $t^\alg{S}(a) = s^\alg{S}(a) = f(a)$.
This proves~\eqref{eq3_thm:FN1_finitely_related}. 
By Lemma~\ref{lma:main_cond_fin_rel} $\alg{S}$ is finitely related.
\end{proof}

\begin{lma}\label{lma:nilp_mon_IMT-fct_form}
Let $\alg{S}$ be a $d$-nilpotent monoid. 
Let $n > |S| + 1$ and let $f \colon S^n \to S$ have \imt{} and depend on all variables. 
Let $e_1,\dotsc,e_n$ be the variable exponents of $f$ and assume that 
\begin{equation}\label{eq0_lma:nilp_mon_IMT-fct_form}
e_1,\dotsc,e_r \le d-2,\quad e_{r+1},\dotsc,e_n \in \{d-1,d\} 
\end{equation}
for some $r \in \{0,\dotsc,n\}$.
Then for all $a_1, \dots, a_n \in S$,
\begin{equation}\label{eq1_lma:nilp_mon_IMT-fct_form}
f(a_1,\dotsc,a_n) = f(a_1,\dotsc,a_r,\bar 1) \,a_{r+1}^{e_{r+1}} \dotsm a_n^{e_n}.
\end{equation}
\end{lma}

\begin{proof}
Fix $a \in S^n$.
Let $c \in \N$ be minimal such that $S \setminus \{1\}$ satisfies $x^c \approx 0$.
By the minimality of $e_i$, as defined in Lemma~\ref{lma:ei_exists}, we have $e_i \le c$ for all $i$.
If $c \le d-2$, then $e_i \le d-2$ for all $i$. 
Then $r = n$ and there is nothing to prove.
Thus assume
\begin{equation}\label{eq2_lma:nilp_mon_IMT-fct_form}
c \in \{d-1,d\}.
\end{equation}
Let $g(a_1,\dotsc,a_n)$ be the right hand side of~\eqref{eq1_lma:nilp_mon_IMT-fct_form}.
If $a_{r+1} = \dotsb = a_n = 1$, then $f(a) = g(a)$. 
Assume $a_k \ne 1$ for some $k \in \{r+1,\ldots,n\}$.
If $a_\ell = 1$ for all $\ell \in \rg{n}\setminus\{k\}$, then $f(a) = a_k^{e_k} = g(a)$.
Assume $a_\ell \ne 1$ for some $\ell \in \rg{n}\setminus\{k\}$. 
Then $g(a) = 0$ since there are at least $d$ factors ${}\ne1$.
Since $n > |S|$, there are $i < j$ in $\rg{n}$ such that $a_i = a_j$.
Let $t$ be a term inducing $f_{ij}$.
Then $f(a) = t^\alg{S}(a)$. 
We show that
\begin{equation}\label{eq3_lma:nilp_mon_IMT-fct_form}
t^\alg{S}(a) = 0. 
\end{equation}
For both $k = i$ and $k \ne i$, the factor $a_k$ occurs at least $c \ge d-1$ times in $t^\alg{S}(a)$ by~\eqref{eq2_lma:nilp_mon_IMT-fct_form} and Lemma~\ref{lma:FN_occurrences}.
Since $a_\ell \ne 1$, there are at least $d$ factors $\ne 1$ in $t^\alg{S}(a)$, and~\eqref{eq3_lma:nilp_mon_IMT-fct_form} follows.
Hence $f(a) = g(a)$, and~\eqref{eq1_lma:nilp_mon_IMT-fct_form} is proved.
\end{proof}

% \begin{customthm}{\ref*{thm:4nilp_fin_rel}}
% Every finite $4$-nilpotent monoid $\alg{S}$ is finitely related with degree at most $|S| + 1$.
% \end{customthm}

\begin{proof}[Proof of Theorem~\ref{thm:4nilp_fin_rel}]
First assume $\alg{S}$ is commutative.
By Theorem~\ref{thm:comm_monoid_fin_rel} $\alg{S}$ is finitely related with degree at most $\max(|S|,3)$.
Since $|S| \ge 2$, the degree is at most $|S|+1$.

Assume $\alg{S}$ is noncommutative.
Let $d \in \N$ be minimal such that $S \setminus \{1\}$ is $d$-nilpotent. 
Let $c \in \mathbb{N}$ be minimal such that $S\setminus \{1\}$ satisfies $x^c \approx 0$.
% NOTE: We can't have $n > |S|$ here, since the existence of the variable exponents in Lemma~\ref{lma:ei_exists} requires $n > |S| + 1$.
Let $n > |S| + 1$ and let $f \colon S^n \to S$ have \imt.
The minimality of $c$ and $d$ and the noncommutativity of $\alg{S}$ yield $d \le |S \setminus \{1\}|$ and
\[
  d \in \{3, 4\}, \quad 2 \le c \le d, \quad d+3 \le n.
\]
By Lemma~\ref{lma:depend_on_each_var} we can assume that $f$ depends on all variables.
Let $e_1,\dotsc,e_n$ be the variable exponents of $f$ defined in Lemma~\ref{lma:ei_exists}.
By reordering variables we assume there are $r\le s$  in $\{0,\dotsc,n\}$ such that
\[
e_1 = \dotsb = e_r = 1, \quad e_{r+1} = \dotsb = e_s = 2, \quad e_{s+1},\dotsc,e_n \ge 3.
\]
We claim the following.
\begin{equation}\label{eq0.7_thm:4nilp_fin_rel}
\begin{aligned}
&\text{Let $a = (a_1, \dots, a_s, \bar 1) \in S^n$ such that $i_1 < \dots < i_m$ in $\rg{s}$ are the}\\
&\text{positions where $a$ is ${} \ne 1$, and $e_{i_1} + \dots + e_{i_m} \ge d$. Then $f(a) = 0$.} 
\end{aligned}
\end{equation}
Since $n > |S|$ there are $i < j$ in $\rg{n}$ with $a_i = a_j$. 
Assume $a_i = a_j = 1$ if $m \le n - 2$ and $a_i = a_j \ne 1$ otherwise.
% Case $m \le n-2$: For all $l \in \rg{s}$, we have $e_l \le 2 \le c$ and thus $\min(e_l, c) = e_l$.
% Case $m \ge n-1$: We have $m \ge d + 2$, thus the sum over the $e_{i_l}$ with $i_l \neq i$ is $\ge d$.
In both cases we obtain $f(a) = f_{ij}(a) = 0$ by Lemma~\ref{lma:FN_occurrences} and $d$-nilpotence.
This proves \eqref{eq0.7_thm:4nilp_fin_rel}.

For $r \ge 1$, we construct a term that induces $f(x_1,\dotsc,x_r,\bar 1)$.
If $r = 1$, this term is $x_1^{e_1}$.
For $r \ge 2$ and $i < j$ in $\rg{r}$, we can identify two $1$'s in
\begin{equation}\label{eq1_thm:4nilp_fin_rel}
f(\bar 1, \underset{i}{x_i}, \bar 1, \underset{j}{x_j}, \bar 1)
\end{equation}
since $n \ge 6$.
Thus~\eqref{eq1_thm:4nilp_fin_rel} is a term function. 
By Lemma~\ref{lma:FN_occurrences} $x_i$ and $x_j$ each occur once in each term inducing~\eqref{eq1_thm:4nilp_fin_rel}.
Since $\alg{S}$ is noncommutative,~\eqref{eq1_thm:4nilp_fin_rel} is induced by either $x_ix_j$ or $x_jx_i$.
This motivates a relation $<$ on $\{x_1,\dotsc,x_r\}$. 
Let
\[\begin{aligned}
  x_i < x_j &\text{ if } f(\bar 1, \underset{i}{x_i}, \bar 1, \underset{j}{x_j}, \bar 1) \text{ is induced by } x_i x_j \text{,} \\
  x_j < x_i &\text{ otherwise.}
\end{aligned}\]
We show that $<$ is transitive. Assume $x_i < x_j$ and $x_j < x_k$ for $i,j,k \in \rg{r}$.
We can identify two $1$'s in the expression
\begin{equation}\label{eq1b_thm:4nilp_fin_rel}
f(\bar 1, \underset{i}{x_i}, \bar 1, \underset{j}{x_j}, \bar 1, \underset{k}{x_k}, \bar 1) \text{.}
\end{equation}
Thus~\eqref{eq1b_thm:4nilp_fin_rel} is a term function induced by some term $t$.
By Lemma~\ref{lma:FN_occurrences} $x_i$, $x_j$, and $x_k$ each occur once in $t$.
We have $t(x_i,x_j,1) = x_i x_j$ and $t(1,x_j,x_k) = x_j x_k$.
Thus $t = x_i x_j x_k$, which implies $x_i < x_k$. 
Now $<$ is a strict linear order on $x_1, \dots, x_r$.
By reordering variables we can assume that $x_1 < \dotsb < x_r$.
Now we claim that
\begin{equation}\label{eq2_thm:4nilp_fin_rel}
x_1\dotsm x_r \text{ induces } f(x_1,\dotsc,x_r,\bar 1).
\end{equation}
Fix $a := (a_1,\dotsc,a_r,\bar 1) \in S^n$.
Let $i_1 < \dotsb < i_m$ in $\rg{r}$ be the positions where $a$ is ${} \ne 1$.
If $m \ge d$, then $f(a) = 0 = a_1 \dotsm a_r$ by~\eqref{eq0.7_thm:4nilp_fin_rel} and $d$-nilpotence.
If $m < d$, then we can identify two $1$'s. 
By \imt{} there is a term $t(x_{i_1},\dotsc,x_{i_m})$ such that
\[
   \text{$t$ induces }
f(\bar{1},\underset{i_1}{x_{i_1}},\bar{1},\underset{i_2}{x_{i_2}},\dotsc,\bar 1,\underset{i_m}{x_{i_m}},\bar 1).
\]
Since $e_1 = \dotsb = e_r = 1$ and by Lemma~\ref{lma:FN_occurrences}, each of the variables $x_{i_1}, \dots, x_{i_m}$ occurs exactly once in $t$.
Now $x_{i_1} < \dotsb < x_{i_m}$ implies $t = x_{i_1} \dotsm x_{i_m}$.
Thus $f(a) = a_{i_1} \dotsm a_{i_m} = a_1 \dotsm a_r$.
This proves~\eqref{eq2_thm:4nilp_fin_rel}. 

If $d=3$, then Lemma~\ref{lma:nilp_mon_IMT-fct_form} implies
\[ f(x_1,\dotsc,x_n) = f(x_1,\dotsc,x_r,\bar 1)\,x_{r+1}^{e_{r+1}} \dotsm x_n^{e_n}, \]
which is a term function by~\eqref{eq2_thm:4nilp_fin_rel}. 

For the rest of the proof assume $d=4$.
% This implies $n\ge 7$.
Lemma~\ref{lma:nilp_mon_IMT-fct_form} implies
\begin{equation}\label{eq0.5_thm:4nilp_fin_rel}
f(x_1,\dotsc,x_n) = f(x_1,\dotsc,x_s,\bar 1) \,x_{s+1}^{e_{s+1}} \dotsm x_n^{e_n}. 
\end{equation}
We claim that
\begin{equation}\label{eq0.6_thm:4nilp_fin_rel}
\text{$f(x_1,\dotsc,x_s,\bar 1)$ is a term function}. 
\end{equation}

If $s \le n-2$, then~\eqref{eq0.6_thm:4nilp_fin_rel} holds since we can identify two $1$'s.
Thus assume $s \ge n-1$ for the rest of the proof.

First assume $r = 0$.
Let $t := x_1^2 \dotsm x_s^2$ and $a = (a_1,\dotsc,a_s,\bar 1) \in S^n$.
We claim that 
\begin{equation}\label{eq3.5_thm:4nilp_fin_rel}
t^\alg{S}(a) = f(a).
\end{equation}
Let $i_1 < \dotsb < i_m$ in $\rg{s}$ be the positions where $a$ is ${} \ne 1$.
If $m \ge 2$, then $f(a) = 0 = t^\alg{S}(a)$ by~\eqref{eq0.7_thm:4nilp_fin_rel} and $4$-nilpotence.
If $m = 1$, then $f(a) = a_{i_1}^2 = t^\alg{S}(a)$ since $e_{i_1} = 2$.
If $m = 0$, then $f(a) = 1 = t^\alg{S}(a)$.
This proves~\eqref{eq3.5_thm:4nilp_fin_rel} and~\eqref{eq0.6_thm:4nilp_fin_rel} for $r = 0$.

Next assume $r > 0$.
For each $i\in\rg{r}$ and $m \in \{r+1,\dotsc,s\}$ we pick a term $w_{im}(x_i,x_m)$ such that
\[
w_{im} \text{ induces } f( \bar 1, \underset{i}{x_i}, \bar 1, \underset{m}{x_m}, \bar 1 ).
\]
By Lemma~\ref{lma:FN_occurrences} $x_i$ occurs once and $x_m$ twice in $w_{im}$.
Thus
\[
w_{im} \in \{ x_i x_m^2, x_m x_i x_m, x_m^2 x_i \}.
\]
Later in the proof, we will extend the term $x_1 \dotsm x_r$ by adding two occurrences of $x_m$ for each $m \in \{r+1,\dotsc,s\}$.

Now we claim that for each $m \in \{r + 1, \dots, s\}$ there exist $\alpha < \beta$ in $\{ 0,\dotsc,r+1 \}$ such that
\begin{equation}\label{eq5.5x_thm:4nilp_fin_rel}
  w_{im} \approx \begin{cases}
  x_i x_m^2   & \text{ if } 1 \le i \le \alpha, \\
  x_m x_i x_m & \text{ if } \alpha < i < \beta, \\
  x_m^2 x_i   & \text{ if } \beta \le i \le r.
  \end{cases}
\end{equation}
We consider two cases, one where $\alg{S}$ satisfies $x y^2 \approx y^2 x$ and another where it does not.

Case~1,
\begin{equation}\label{case1_eq4_thm:4nilp_fin_rel}
  \text{$\mathbf{S}$ does not satisfy $x y^2 \approx y^2 x$.}
\end{equation}
For $i < j$ in $\rg{r}$ and $m \in \{r+1,\dotsc,s\}$ we claim that
\begin{align}
\label{case1_claim1_eq5_thm:4nilp_fin_rel}
&\text{if $w_{im} \approx x_m^2 x_i$, then $w_{jm} \approx x_m^2 x_j$,} \\
\label{case1_claim2_eq5_thm:4nilp_fin_rel}
&\text{if $w_{jm} \approx x_j x_m^2$, then $w_{im} \approx x_i x_m^2$.}
\end{align}
To prove~\eqref{case1_claim1_eq5_thm:4nilp_fin_rel} assume that $w_{im} \approx x_m^2 x_i$.
By \imt{} there is a term $v(x_i, x_j, x_m)$ such that
\[
  \text{$v$ induces $f( \bar 1, \underset{i}{x_i}, \bar 1, \underset{j}{x_j}, \bar 1, \underset{m}{x_m}, \bar 1 )$.}
\]
By the definition of $w_{im}$ and~\eqref{case1_eq4_thm:4nilp_fin_rel} we have
\begin{align}\label{foo41}
  v(x_i,1,x_m) \approx w_{im} &\approx x_m^2 x_i \not\approx x_i x_m^2 \text{.}
\end{align}
Since $x_i < x_j$ we have
\begin{equation}\label{foo42}
  v(x_i, x_j, 1) = x_i x_j \text{.}
\end{equation}
By Lemma~\ref{lma:FN_occurrences} the variables $x_i,x_j$ occur once and $x_m$ twice in $v$.
From~\eqref{foo41} and~\eqref{foo42} it follows that
\begin{equation}\label{foo43}
  v \in \{x_m x_i x_j x_m, x_m x_i x_m x_j, x_m^2 x_i x_j\} \text{.}
\end{equation}
If $v = x_m x_i x_j x_m$, then
\[
  x_m^2 x_i \approx w_{im} \approx v(x_i, 1, x_m) = x_m x_i x_m
\]
and thus $\alg{S}$ satisfies $x^2 y \approx xyx$.
Then
\[
  w_{jm} \approx v(1, x_j, x_m) = x_m x_j x_m \approx x_m^2 x_j \text{.}
\]
On the other hand, if $v$ is one of the last two terms in~\eqref{foo43}, we directly get
\[
  w_{jm} \approx v(1, x_j, x_m) = x_m^2 x_j \text{.}
\]
This proves~\eqref{case1_claim1_eq5_thm:4nilp_fin_rel}.
The proof of~\eqref{case1_claim2_eq5_thm:4nilp_fin_rel} follows from a symmetric argument.

To prove~\eqref{eq5.5x_thm:4nilp_fin_rel}, let $m \in \{r + 1, \dots, s\}$. Let $\alpha \in[r]$ be maximal such that $w_{\alpha m} \approx x_\alpha x_m^2$ and set $\alpha=0$ if no such $\alpha \in[r]$ exists. 
Then $w_{i m} \approx x_i x_m^2$ for all $i \leq \alpha$ by~\eqref{case1_claim2_eq5_thm:4nilp_fin_rel}.
Let $\beta \in[r]$ be minimal such that $w_{\beta m} \approx x_m^2 x_\beta$ and set $\beta=r+1$ if no such $\beta \in[r]$ exists.
Then $w_{i m} \approx x_m^2 x_i$ for $\beta \leq i \leq r$ by \eqref{case1_claim1_eq5_thm:4nilp_fin_rel}. Since $x y^2 \not \approx y^2 x$ in $\mathbf{S}$, we have $\alpha<\beta$ and $w_{i m} \approx x_m x_i x_m$ for $\alpha<i<\beta$.
This proves~\eqref{eq5.5x_thm:4nilp_fin_rel} for case~1.

Case~2,
\begin{equation}\label{case2_eq4_thm:4nilp_fin_rel}
  \text{$\mathbf{S}$ satisfies $x y^2 \approx y^2 x$.}
\end{equation}
For $i < j < k$ in $\rg{r}$ and $m \in \{r+1,\dotsc,s\}$ we claim that
\begin{equation}\label{foo44}
  \text{if $w_{im} \approx x_m x_i x_m$ and $w_{km} \approx x_m x_k x_m$, then $w_{jm} \approx x_m x_j x_m$.}
\end{equation}
Assume $w_{im} \approx x_m x_i x_m$ and $w_{km} \approx x_m x_k x_m$.
Since $n \ge 6$ we can identify two $1$'s in
\begin{equation}\label{foo46}
  f( \bar 1, \underset{i}{x_i}, \bar 1, \underset{j}{x_j}, \bar 1, \underset{k}{x_k}, \bar 1, \underset{m}{x_m}, \bar 1 ) \text{.}
\end{equation}
By \imt{} there is a term $v(x_i,x_j,x_k,x_m)$ which induces~\eqref{foo46}.
By Lemma~\ref{lma:FN_occurrences} the variables $x_i,x_j,x_k$ occur once and $x_m$ twice in $v$.
Since $x_i < x_j < x_k$, it follows that $x_i$ occurs before $x_j$ and $x_j$ before $x_k$ in $v$.
Observe that
\begin{equation}\label{foo45}
  \begin{aligned}
  v(x_i,x_j,x_k,1) &= x_i x_j x_k, \\
  v(x_i,1,1,x_m) &\approx w_{im} \approx x_m x_i x_m, \\
  v(1,1,x_k,x_m) &\approx w_{km} \approx x_m x_k x_m.
  \end{aligned}
\end{equation}
If $xyx \not\approx x^2y \approx yx^2$ in $\alg{S}$, then $v = x_m x_i x_j x_k x_m$.
If $xyx \approx x^2y \approx yx^2$, then the $x_m$ can be moved to any position in~\eqref{foo45}.
Both scenarios yield $v \approx x_m x_i x_j x_k x_m$.
We obtain
\[
w_{jm} \approx v(1,x_j,1,x_m) \approx x_m x_j x_m \text{,}
\]
which proves~\eqref{foo44}.

To prove~\eqref{eq5.5x_thm:4nilp_fin_rel}, let $m \in \{r + 1, \dots, s\}$.
If $w_{im} \not\approx x_m x_i x_m$ for all $i \in \rg{r}$, set $\alpha =0$ and $\beta = 1$.
Then~\eqref{case2_eq4_thm:4nilp_fin_rel} implies~\eqref{eq5.5x_thm:4nilp_fin_rel}.
If $w_{im} \approx x_m x_i x_m$ for some $i \in \rg{r}$, let $\alpha \in \{0, \dots, r-1\}$ be minimal such that $w_{\alpha+1, m} \approx x_m x_{\alpha+1} x_m$ and let $\beta \in \{2, \dots, r+1\}$ be maximal such that $w_{\beta-1, m} \approx x_m x_{\beta-1} x_m$.
Then~\eqref{foo44} yields $w_{i m} \approx x_m x_i x_m$ for $\alpha<i<\beta$.
This and~\eqref{case2_eq4_thm:4nilp_fin_rel} yield~\eqref{eq5.5x_thm:4nilp_fin_rel}.

% For each $m \in \{r + 1, \dots, s\}$ let $\alpha \in \{0, \dots, r\}$ be minimal such that $w_{\alpha+1, m} \approx x_m x_{\alpha+1} x_m$ and let $\beta \in \{1, \dots, r + 1\}$ be maximal such that $w_{\beta-1, m} \approx x_m x_{\beta-1} x_m$. Then~\eqref{foo44} yields $w_{i m} \approx x_m x_i x_m$ for all $\alpha<i<\beta$.
% This and~\eqref{case2_eq4_thm:4nilp_fin_rel} yield~\eqref{eq5.5x_thm:4nilp_fin_rel} for case~2.

Using~\eqref{eq5.5x_thm:4nilp_fin_rel} we construct a term $t$ that induces $f(x_1,\dotsc,x_s,\bar 1)$.
For each \( m \in \{r + 1, \dots, s\} \), we insert two occurrences of \( x_m \) into \( x_1 \dotsm x_r \) to get
\[
  x_1 \dotsm x_\alpha x_m x_{\alpha + 1} \dotsm x_{\beta - 1} x_m x_\beta \dotsm x_r \text{.}
\]
We repeat this procedure for all such $m$ and denote the final term by $t$.
Note that for different $m$ the new variables can be inserted independently from each other since all permutations of $x_{r+1}^2 \dotsm x_s^2$ are equivalent terms in $\alg{S}$ by $4$-nilpotence.
For $i \in \rg{r}$ and $m \in \{r+1,\dotsc,s\}$ we obtain
\begin{equation}\label{eq6_thm:4nilp_fin_rel}
t(\bar 1, \underset{i}{x_i}, \bar 1, \underset{m}{x_m}, \bar 1) \approx w_{im}
\end{equation}
from~\eqref{eq5.5x_thm:4nilp_fin_rel}.
For $a = (a_1,\dotsc,a_s,\bar 1) \in S^n$, we claim that 
\begin{equation}\label{eq7x_thm:4nilp_fin_rel}
t^\alg{S}(a) = f(a).
\end{equation}
Let $i_1 < \dotsb < i_\ell$ in $\rg{s}$ be the positions where $a$ is ${} \ne 1$.
If $e_{i_1} + \dotsb + e_{i_\ell} \ge 4$, then $f(a) = 0 = t^\alg{S}(a)$ by~\eqref{eq0.7_thm:4nilp_fin_rel} and $4$-nilpotence.
If $a_{r+1} = \dots = a_s = 1$, then $f(a) = a_1 \dots a_r = t^\alg{S}(a)$ by~\eqref{eq2_thm:4nilp_fin_rel}.
If $e_{i_1} + \dotsb + e_{i_\ell} < 4$ and $a_m \ne 1$ for some $m \in \{r + 1, \dots, s\}$, then $a$ is of the form 
\[ a = (\bar 1, \underset{i}{a_i}, \bar 1, \underset{m}{a_m}, \bar 1) 
\quad\text{for some $i \in \rg{r}$ and $m \in \{r+1,\dotsc,s\}$.} \]
Now~\eqref{eq6_thm:4nilp_fin_rel} yields $t^\alg{S}(a) = w_{im}^\alg{S}(a_i,a_m) = f(a)$.
This proves~\eqref{eq7x_thm:4nilp_fin_rel} and~\eqref{eq0.6_thm:4nilp_fin_rel}.
By~\eqref{eq0.5_thm:4nilp_fin_rel} $f$ is a term function.
By Lemma~\ref{lma:main_cond_fin_rel} $\alg{S}$ is finitely related with degree at most $|S| + 1$.
\end{proof}

% \begin{customthm}{\ref*{thm:FN5_factor_not_fr}}%\label{thm:FN5_factor_not_fr}
% Let $\alg{S} := (\algfn_5 \{a,b\} \mathbin{/}\theta)^1$, where $\theta$ is the congruence whose equivalence classes are $\{ (ab)^2, (ba)^2 \}$, $\{ a^2b^2, b^2a^2 \}$, and the rest singletons.
% Then $\alg{S}$ is not finitely related.
% \end{customthm}

\begin{proof}[Proof of Theorem~\ref{thm:FN5_factor_not_fr}]
The relation $\theta$ is a congruence on $\algfn_5 \{a,b\}$ since the elements in the nontrivial equivalence classes have $4$ factors and $\algfn_5 \{a,b\}$ is $5$-nilpotent.
Furthermore,
\begin{equation}\label{eq0_thm:FN5_factor_not_fr}
\text{$\alg{S}$ satisfies $(xy)^2\approx (yx)^2$ and $x^2y^2 \approx y^2x^2$.}
\end{equation}
For $n \ge 4$ we define $f \colon S^n \to S$ as follows.
Let 
\[ f(x_1,\dotsc,x_n) := 0 \quad\text{if at least three $x_i$ are ${} \ne 1$}, \]
and for $i < j$ in $\rg{n}$ let
\[
f(\bar 1, \underset{i}{x_i}, \bar 1, \underset{j}{x_j}, \bar 1) := 
\begin{cases}
(x_i x_j)^2 &\text{if } j = i+1 \text{ or } (i,j)=(1,n), \\
x_i^2 x_j^2 &\text{otherwise.}
\end{cases}
\]
First we show that 
\begin{equation}\label{eq1_thm:FN5_factor_not_fr}
f \text{ has \imt.}
\end{equation}
%
% For $i,j \in \rg{n}$ define a term $g_{ij}$ by
% \[
% g_{ij} := 
% \begin{cases}
% x_i^2 &\text{if } i=j, \\
% %
% x_i x_{i \oplus 1} x_i x_{i \oplus 2} x_{i \oplus 1} x_{i \oplus 3} x_{i \oplus 2} \dotsm x_{j \ominus 1} x_{j \ominus 2} x_j x_{j \ominus 1} x_j \hspace{1cm}
% &\text{if } i \ne j.
% \end{cases}
% \]
%
For $i,j \in \rg{n}$ let $\ell := (j - i \mod n) + 1$.
We define a term $g_{ij}$ of length $2 \ell$.
Let $\oplus$ and $\ominus$ be addition and subtraction modulo $n$ on the set $\rg{n}$, respectively.
Let
\[
  g_{ii} := x_i ^2.
\]
For $i \ne j$ let $g_{ij}^{(k)}$ denote the variable in the $k$-th position of $g_{ij}$ and
\[
\begin{aligned}
& g_{ij}^{(1)} := g_{ij}^{(3)} := x_i \\
& g_{ij}^{(2 k)} := g_{ij}^{(2 k + 3)} := x_{i \oplus k} && \text{ for } 0 < k < \ell - 1, \\
& g_{ij}^{(2 \ell - 2)} := g_{ij}^{(2 \ell)} := x_j.
\end{aligned}
\]
Note that if $i < j$, then $g_{ij}$ contains all variables in $\{x_i, \dots, x_j\}$ exactly twice.
Conversely, if $i > j$, then $g_{ij}$ contains all variables in $\{x_i, \dots, x_n, x_1, \dots, x_j\}$ exactly twice.
For example, for $n=5$ we have $g_{42} = x_4 x_5 x_4 x_1 x_5 x_2 x_1 x_2$.
% \begin{align*} 
% g_{12} &= (x_1 x_2)^2, \\
% g_{14} &= x_1 x_2 x_1 x_3 x_2 x_4 x_3 x_4, \\
% g_{42} &= x_4 x_5 x_4 x_1 x_5 x_2 x_1 x_2. \\
% \end{align*}

Now let $i < j$ in $\rg{n}$ and 
\[
\begin{aligned}
h &:= \begin{cases}
g_{2,n-1} &\text{if } (i,j)=(1,n), \\
g_{j\oplus1,i\ominus1} &\text{if } j=i+1, \\
g_{j\oplus1,i\ominus1} \, g_{i+1,j-1} &\text{otherwise,}
\end{cases} \\
r &:= x_j^4 h.
\end{aligned}
\]
We claim that
\begin{equation}\label{eq2_thm:FN5_factor_not_fr}
\text{$r$ induces $f_{ij}$.}
\end{equation}
Note that $x_k$ occurs twice in $h$ for $k \in \rg{n} \setminus \{i,j\}$, whereas $x_i$ and $x_j$ do not occur. 
Fix $i < j$ in $\rg{n}$ and $c \in S^n$.

Case, $c_j \ne 1$ and $c_k = 1$ for all $k \in \rg{n}\setminus\{i,j\}$.
Then $h(c) = 1$ and $r(c) = c_j^4 = f_{ij}(c)$.

Case, $c_j \ne 1$ and $c_k \ne 1$ for some $k \in \rg{n}\setminus\{i,j\}$.
Then $h(c) \ne 1$ and $r(c) = c_j^4 \, h(c) = 0 = f_{ij}(c)$ by $5$-nilpotence.

Case, $c_j = 1$ and at least $3$ positions of $c$ among $\rg{n} \setminus \{i,j\}$ are ${} \ne 1$.
Then $h(c) = 0$ by $5$-nilpotence. Thus $r(c) = 0 = f_{ij}(c)$.

Case, $c_j = 1$ and at most $2$ positions of $c$ among $\rg{n} \setminus \{i,j\}$ are ${} \ne 1$.
Pick $k < \ell$ in $\rg{n} \setminus \{i,j\}$ such that all positions of $c$ among $\rg{n} \setminus \{i,j,k,\ell\}$ are~$1$.
Let $A := \{i+1, \dots, j-1\}$ and $B := \{1, \dots, i - 1, j + 1, \dots, n\}$.
If $k, \ell \in A$, then $h(c) = g_{i + 1, j - 1}(c)$.
If $k, \ell \in B$, then $h(c) = g_{j \oplus 1, i \ominus 1}(c)$.
If $k \in A$ and $\ell \in B$, then $g_{j \oplus 1, i \ominus 1}(c) = c_\ell^2$ and $g_{i + 1, j - 1}(c) = c_k^2$, and thus $h(c) = c_\ell^2 c_k^2$.
If $k \in B$ and $\ell \in A$, then similarly $h(c) = c_k^2 c_\ell^2$.
By the identities~\eqref{eq0_thm:FN5_factor_not_fr} we have
\begin{equation}\label{eq2.5_thm:FN5_factor_not_fr}
r(c) = h(c) =
\begin{cases}
(c_k c_\ell)^2 &\text{if } \ell=k \oplus 1 \text{ or } k=\ell \oplus 1, \\
%c_k c_\ell c_k c_\ell &\text{if } (k,\ell)=(1,n), \\
c_k^2 c_\ell^2 &\text{otherwise.}
\end{cases}
\end{equation}
By definition,
\[
f_{ij}(c) = f(\bar 1, \underset{k}{c_k}, \bar 1, \underset{\ell}{c_\ell}, \bar 1)
\]
is also equal to ~\eqref{eq2.5_thm:FN5_factor_not_fr}.
Thus $f_{ij}(c) = r(c)$.
This proves~\eqref{eq2_thm:FN5_factor_not_fr}.
Thus every $f_{ij}$ is a term function, and~\eqref{eq1_thm:FN5_factor_not_fr} follows.

Next we claim that
\begin{equation}\label{eq3_thm:FN5_factor_not_fr}
f \text{ is not a term function.}
\end{equation}
Suppose $f$ is induced by some term $t$.
Recall that $a \in S$.
Every variable $x_i$ occurs twice in $t$ since
\[
t^\alg{S}(\bar 1, \underset{i}{a}, \bar 1) = f(\bar 1, \underset{i}{a}, \bar 1) = a^2 \notin \{a,a^3,a^4,0\}.
\]
By~\eqref{eq0_thm:FN5_factor_not_fr} and the definition of $f$, we have $t^{\alg{S}}(x_1,\dotsc,x_n) = t^{\alg{S}}(x_2,\dotsc,x_n,x_1)$.
Thus we can assume that $t$ starts with $x_2$.
By the definition of $f$, the following equivalences hold in $\alg{S}$.
\begin{align*}
\begin{aligned}
t(x_1,x_2,\bar 1) &\approx (x_1 x_2)^2, \\
t(1,x_2,x_3,\bar 1) &\approx (x_2 x_3)^2, \\
t(1,x_2,\bar 1, \underset{i}{x_i},\bar 1) &\approx x_2^2 x_i^2 &&\text{for } i \in \{4,\dotsc,n\}.
\end{aligned}
\end{align*}
Thus $x_1$ and $x_3$ each occur once between the two occurrences of $x_2$ in $t$, whereas $x_4,\dotsc,x_n$ do not occur between the two $x_2$'s.
So $t$ starts either with $x_2 x_1 x_3 x_2$ or with $x_2 x_3 x_1 x_2$.
Hence
\[
t(x_1,1,x_3,\bar 1) \in \{ (x_1 x_3)^2, (x_3 x_1)^2, x_1 x_3^2 x_1, x_3 x_1^2 x_3 \}.
\]
So $t(x_1,1,x_3,\bar 1)$ and $x_1^2 x_3^2$ cannot be equivalent.
Since $n \ge 4$, this contradicts the definition of $f$.
We proved~\eqref{eq3_thm:FN5_factor_not_fr}.
By~\eqref{eq1_thm:FN5_factor_not_fr},~\eqref{eq3_thm:FN5_factor_not_fr}, and Lemma~\ref{lma:main_cond_fin_rel}, $\alg{S}$ is not finitely related.
\end{proof}

The following was proved independently by the authors of~\cite{DJPS11} and Markovi\'c, Mar\'oti, and McKenzie in~\cite{MMM12}.

\begin{thm}[{
\cite[Theorem~2.11]{DJPS11}, \cite[Corollary~4.3]{MMM12}
}]
\label{thm:gen_same_variety}
Let $\alg{A},\alg{B}$ be finite algebras that generate the same variety. Then $\alg{A}$ is finitely related if and only if $\alg{B}$ is.
\end{thm}

\begin{lma}\label{lma:fin_rel_propagation}
For a variety $V$, the following two conditions are equivalent:
\begin{enumerate}
\item\label{it1_lma:fin_rel_propagation}
$\Forall \alg{A}, \alg{B} \in V_\textnormal{fin}$: if $\alg{A} \times \alg{B}$ is finitely related (f.r.), then $\alg{A}$ is f.r.
\item\label{it2_lma:fin_rel_propagation}
$\Forall \alg{A} \in V_\textnormal{fin}$: if $\alg{A}$ is f.r., then every homomorphic image of $\alg{A}$ is f.r.
\end{enumerate}
These conditions imply the following:
\begin{enumerate}
\setcounter{enumi}{2}
\item\label{it3_lma:fin_rel_propagation}
$\Forall \alg{A} \in V_\textnormal{fin}$: if $\alg{A}$ is f.r., then every subalgebra of $\alg{A}$ is f.r.
\end{enumerate}
\end{lma}

\begin{proof}
Clearly~\enumref{it2_lma:fin_rel_propagation} implies~\enumref{it1_lma:fin_rel_propagation}.
Assume~\enumref{it1_lma:fin_rel_propagation} holds. 
Let $\alg{B}$ be a homomorphic image of $\alg{A}$. Then $\alg{A}$ and $\alg{A} \times \alg{B}$ generate the same variety. 
By Theorem~\ref{thm:gen_same_variety} $\alg{A} \times \alg{B}$ is also finitely related.
By~\enumref{it1_lma:fin_rel_propagation} $\alg{B}$ is finitely related.
This proves~\enumref{it2_lma:fin_rel_propagation}.
The fact that~\enumref{it1_lma:fin_rel_propagation} implies~\enumref{it3_lma:fin_rel_propagation} is proved similarly.
\end{proof}

For a semigroup $\alg{S}$ let $\alg{S} \sqcup \{0\}$ denote the semigroup obtained when we adjoin a (possibly additional) zero element, where we extend the multiplication by $s \cdot 0 := 0 \cdot s := 0 \cdot 0 := 0$ for all $s \in S$.

\begin{thm}[{\cite[Theorem~4.1]{May13}}]\label{thm:adj_zero_fin_rel}
A finite semigroup $\alg{S}$ is finitely related if and only if $\alg{S} \sqcup \{0\}$ is finitely related.
\end{thm}

Let $\alg{S}^0$ denote the semigroup
\[
\alg{S}^0 := \begin{cases}
\alg{S} \sqcup \{0\} &\text{if $\alg{S}$ has no zero element,} \\
\alg{S} &\text{otherwise,}
\end{cases}
\]
For two semigroups $\alg{S}$ and $\alg{T}$ with $S^0 \cap T^0 = \{0\}$, we define the \emph{$0$-direct union} $\alg{S} \cupz \alg{T}$ to be the semigroup with universe $S^0 \cup T^0$ and multiplication
\[
x \cdot y := \begin{cases}
x \cdot y &\text{if } x, y \in S^0 \text{ or } x,y \in T^0, \\
0 &\text{otherwise.}
\end{cases}
\]
The variety generated by an algebra $\alg{A}$ is denoted by $\vg \alg{A}$.

\begin{lma}\label{sg_vars_gen}
For finite semigroups $\alg{S}$ and $\alg{T}$ we have
$
\vg (\alg{S} \times \alg{T})^0 = \vg \alg{S}^0 \times \alg{T}^0  = \vg \alg{S} \cupz \alg{T}.
$
\end{lma}

\begin{proof}
Since $\alg{S}$ and $\alg{T}$ are finite, both have an idempotent element.
For the first equality observe that both $\alg{S}^0$ and $\alg{T}^0$ embed into $(\alg{S} \times \alg{T})^0$.
Conversely $(\alg{S} \times \alg{T})^0$ is a subsemigroup of $\alg{S}^0 \times \alg{T}^0$.

We show the second equality. Both $\alg{S}^0$ and $\alg{T}^0$ are subsemigroups of $\alg{S} \cupz \alg{T}$.
Conversely $\alg{S} \cupz \alg{T}$ embeds into $\alg{S}^0 \times \alg{T}^0$ via $S \to S \times \{0\}$, $T \to \{0\} \times T$, and $0 \mapsto (0,0)$.
\end{proof}

\begin{lma}\label{lma:zerodirect_union_fin_rel}
For finite semigroups $\alg{S}$ and $\alg{T}$, $\alg{S} \times \alg{T}$ is finitely related if and only if $\alg{S} \cupz \alg{T}$ is finitely related.
\end{lma}

\begin{proof}
By Theorem~\ref{thm:adj_zero_fin_rel} $\alg{S} \times \alg{T}$ is finitely related if and only if $(\alg{S} \times \alg{T})^0$ is finitely related. Now the result is immediate from Theorem~\ref{thm:gen_same_variety} and Lemma~\ref{sg_vars_gen}.
\end{proof}

% \begin{customthm}{\ref*{thm:fin_rel_propagation_sg}}
% \textcolor{Gray}{
% For a variety of semigroups $V$, the following are equivalent:
% \begin{enumerate}
% \item\label{it1_thm:fin_rel_propagation_sg}
% $\Forall \alg{S}, \alg{T} \in V_\textnormal{fin}$: if $\alg{S} \times \alg{T}$ is finitely related (f.r.), then $\alg{S}$ is f.r.
% \item\label{it2_thm:fin_rel_propagation_sg}
% $\Forall \alg{S} \in V_\textnormal{fin}$: if $\alg{S}$ is f.r., then every homomorphic image of $\alg{S}$ is f.r.
% \setcounter{enumi}{2}
% \item\label{it3_thm:fin_rel_propagation_sg}
% $\Forall \alg{S} \in V_\textnormal{fin}$: if $\alg{S}$ is f.r., then every subsemigroup of $\alg{S}$ is f.r.
% \item\label{it4_thm:fin_rel_propagation_sg}
% $\Forall \alg{S} \in V_\textnormal{fin}$: if $\alg{S}$ is f.r., then every Rees quotient of $\alg{S}$ is f.r.
% \end{enumerate}}
% \end{customthm}

\begin{proof}[Proof of Theorem~\ref{thm:fin_rel_propagation_sg}]
By Lemma~\ref{lma:fin_rel_propagation} it suffices to show that~\enumref{it3_thm:fin_rel_propagation_sg} implies~\enumref{it1_thm:fin_rel_propagation_sg} and that \enumref{it4_thm:fin_rel_propagation_sg} implies~\enumref{it3_thm:fin_rel_propagation_sg}.
First assume~\enumref{it3_thm:fin_rel_propagation_sg} holds and $\alg{S} \times \alg{T}$ is finitely related.
Since $\alg{T}$ is finite, it contains an idempotent element $e$. Thus $\alg{S} \cong \alg{S} \times \{e\}$ embeds into $\alg{S} \times \alg{T}$. 
So $\alg{S}$ is finitely related.

Now assume~\enumref{it4_thm:fin_rel_propagation_sg} holds. Let $\alg{T}$ be a subsemigroup of a finitely related semigroup $\alg{S}$.
Since $\alg{S}$ and $\alg{S} \times \alg{T}$ generate the same variety, $\alg{S} \times \alg{T}$ is also finitely related by Theorem~\ref{thm:gen_same_variety}.
By Lemma~\ref{lma:zerodirect_union_fin_rel} $\alg{S} \cupz \alg{T}$ is finitely related.
Now $\alg{S}^0$ is an ideal of $\alg{S} \cupz \alg{T}$. The corresponding Rees quotient is isomorphic to $\alg{T}^0$ and finitely related by~\enumref{it4_thm:fin_rel_propagation_sg}.
By Theorem~\ref{thm:adj_zero_fin_rel} $\alg{T}$ is finitely related.
\end{proof}

\begin{lma}\label{lma:examples}
Let $\alg{S} := (\algfn_5 \{a,b\} \mathbin{/}\theta)^1$ be as in Theorem~\ref{thm:FN5_factor_not_fr} and $\alg{T} \cong (\algfn_5 \{a, b\})^{1}$ with $S$ and $T$ disjoint.
Then
\begin{enumerate}
  \item $\alg{S} \cupz \alg{T}$ is finitely related.
  \item $\alg{S}$ is a subsemigroup and a Rees quotient of $\alg{S} \cupz \alg{T}$ but not finitely related.
\end{enumerate}
\end{lma}

\begin{proof}
By Lemma~\ref{sg_vars_gen} $\alg{T}$ and $\alg{S} \cupz \alg{T}$ generate the same variety. 
Thus $\alg{S} \cupz \alg{T}$ is finitely related by Theorem~\ref{thm:FN1_finitely_related} and Theorem~\ref{thm:gen_same_variety}.
Since $\alg{T}$ forms an ideal in $\alg{S} \cupz \alg{T}$ we have $\alg{S} \cong (\alg{S} \cupz \alg{T}) \mathbin{/} \alg{T}$.
By Theorem~\ref{thm:FN5_factor_not_fr} $\alg{S}$ is not finitely related.
\end{proof}

\end{document}